\documentclass[12pt,reqno]{article}
\usepackage[usenames]{color}
\usepackage{graphicx}
\usepackage{amscd}
\usepackage{amsfonts}
\usepackage{mathrsfs}

\usepackage[colorlinks=true,
linkcolor=webgreen,
filecolor=webbrown,
citecolor=webgreen]{hyperref}

\definecolor{webgreen}{rgb}{0,.5,0}
\definecolor{webbrown}{rgb}{.6,0,0}

\usepackage{color}
\usepackage{fullpage}
\usepackage{float}

\usepackage{graphics,amsmath,amssymb}
\usepackage{amsthm}
\usepackage{latexsym}
\usepackage{epsf}

\let\mathnumsetfont\mathbb

\newcommand{\seqnum}[1]{\href{http://oeis.org/#1}{\underline{#1}}}

\begin{document}


\theoremstyle{plain}
\newtheorem{theorem}{Theorem}
\newtheorem{corollary}[theorem]{Corollary}
\newtheorem{lemma}[theorem]{Lemma}
\newtheorem{proposition}[theorem]{Proposition}

\theoremstyle{definition}
\newtheorem{definition}[theorem]{Definition}
\newtheorem{example}[theorem]{Example}
\newtheorem{conjecture}[theorem]{Conjecture}

\theoremstyle{remark}
\newtheorem{remark}[theorem]{Remark}

\begin{center}
\vskip 1cm{\LARGE\bf 
Lambda Words: A Class of Rich Words \\
\vskip .12in
Defined Over an Infinite Alphabet
}
\vskip 1cm
\large
Norman Carey\\
CUNY Graduate Center\\
365 Fifth Avenue\\
New York, NY 10016\\
USA\\
\href{mailto:ncarey@gc.cuny.edu}{\tt ncarey@gc.cuny.edu}
\end{center}

\vskip .2 in

\newcommand\Nset{\mathnumsetfont N}
\newcommand\Qset{\mathnumsetfont Q}
\newcommand\Rset{\mathnumsetfont R}
\newcommand\Zset{\mathnumsetfont Z}

\begin{abstract}
Lambda words are sequences obtained by encoding the differences between ordered elements of the form $i+j\theta$, where $i$ and $j$ are non-negative integers and $1 < \theta <2$. Lambda words are right-infinite words defined over an infinite alphabet that have connections with Sturmian words, Christoffel words, and interspersion arrays. We show that Lambda words are infinite rich words. Furthermore, any Lambda word may be mapped onto a right-infinite word over a three-letter alphabet. Although the mapping preserves palindromes and non-palindromes of the Lambda word, the resulting Gamma word is not rich.
\end{abstract}
 
\section{Introduction}

Here we formalize the {\it Lambda word} \cite{car11}, which is a sequence formed by encoding the differences between ordered elements of the form $i+j\theta$, where $i$ and $j$ are non-negative integers and $\theta$ is irrational, $1<\theta <2$. For example, when $\theta = \phi = \frac{\sqrt{5}+1}{2}$, the ordering begins 
\[
(0+0\phi) < (1+0\phi) < (0+1\phi) < (2+0\phi) <(1+1\phi) <\cdots.
\]

 Taking differences from $(0,1,\phi,2,1+\phi,\ldots)$ we get 
\[
1 ,\ \phi -1,\ 2-\phi,\ \phi-1, \ldots.
\]

There are four differences so far, but two of them are the same, and so the sequence of differences may be represented by the integer sequence, $(0,1,2,1, \ldots)$. 
This is the beginning of $\Lambda_{\phi}$, the Lambda word generated by $\phi$ (\seqnum{A216763}). Showing more terms, 
\[
\Lambda_{\phi} = (0, 1, 2, 1, 2, 3, 2, 2, 3, 2, 3, 4, 3, 2, 3, 4, 3, 3, 4, 3, 4,\ldots).
\]
A Lambda word is defined over an infinite alphabet $\mathscr{A}$ 
of non-negative integers. Lambda words are related to Christoffel words
\cite{ber08} and Sturmian words \cite{ber02, carp05}. As we will show, they may be derived from interspersion arrays \cite{kim93}. Lambda words developed from considerations of certain palindromic tonal spaces that contain important musical scales \cite{car96}. The main result of this paper is to show that Lambda words are rich words \cite{gle09b,buc09b}.

Lambda words exhibit the following properties:

\begin{enumerate}
\item A Lambda word is a right-infinite word defined over an infinite alphabet \cite[p.\ 46]{car11}.
\item There are no recurrent letters ({\it a fortiori}, no recurrent factors) in a Lambda word.\label{norecur}
\item The number of occurrences of any letter in a Lambda word may be calculated by means of convergents in the continued fraction expansion of $\theta$.\label{cfr}
\item Palindromes in a Lambda word are on alphabets of no more than three letters.\label{no4}
\item A Lambda word is {rich}.\label{lamrich}
\item There exists a projection that maps a Lambda word onto a right-infinite word over a three-letter alphabet that preserves palindromes and non-palindromes.\label{goesto3}
\end{enumerate}

In this paper, after definitions in \S \ref{defsec}, the structure of the Lambda word is outlined through examples in \S \ref{exsec}. In \S \ref{cfsec}, properties of continued fractions are discussed. In \S \ref{countsec}, we demonstrate statements (\ref{norecur}) and (\ref{cfr}), and in \S \ref{langsec}, statement (\ref{no4}). The central result of the paper is to prove statement (\ref{lamrich}) in \S \ref{richsec}. Statement (\ref{goesto3}) is demonstrated in \S \ref{3alphsec}.
We define Lambda words after some basic notation and definitions.

\section{The Lambda word}\label{defsec}

Following Lothaire \cite{lot02}, we say that $\mathscr{A}$ is a set of symbols, finite or infinite. The symbols are referred to as {\it letters}. A finite or infinite sequence of letters from $\mathscr{A}$ is a {\it word}. A finite word over the alphabet $\mathscr{A}$ is an element of $\mathscr{A}^{*}$, the free monoid generated by $\mathscr{A}$. The monoid operation is concatenation where it is understood that the empty word, $\epsilon$, is a member of $\mathscr{A}^{*}$ and serves as the identity element. If $w$ is a word over $\mathscr{A}^{*}$ and $w = uv$, then $u$ is a {\it prefix} and $v$ a {\it suffix} respectively of $w$. Both $u$ and $v$ are {\it factors} of $w$, as is $x$ if $w = uxv$. The set of factors of a finite or infinite word $w$ is denoted by $\mathcal{F}(w)$.
A {\it right-infinite word} over $\mathscr{A}$ is a map $h$ from the set of non-negative integers into $\mathscr{A}$ forming an infinite sequence, $h(0), h(1),\ldots, h(n), \ldots$, written with or without commas.
If $w = x_{1}x_{2}\cdots x_{k}$ where $ x_{1},x_{2},\ldots x_{k} \in \mathscr{A}$ then the {\it reversal} of $w$ is $\widetilde{w} = x_{k}\cdots x_{2}x_{1}$. A {\it palindrome} is a word $p$ such that $p = \widetilde{p}$.

In what follows, $\mathscr{A}$ denotes an infinite alphabet of non-negative integers, elements of $\Nset_{0} = \{0,1,2,3,\ldots\}$. For letters $x,y \in \mathscr{A}$ we write $x\prec y$ to indicate lexicographic order, which, here, coincides with numerical order.

We define $\theta \in \Rset \setminus \Qset$, $1<\theta<2$ and the set 
\begin{displaymath}
S(\theta) = \{i + j\theta \mid i,j \in \Nset_{0}\}.
\end{displaymath}
$S_{\theta}$ will denote the strictly increasing sequence obtained by sorting the elements of $S(\theta)$ in ascending order. The $n$-th term of $S_{\theta}$ is denoted by $s_{n} = i_{n} + j_{n} \theta$. 

The first-order difference sequence of a given sequence $A$ is written as $\Delta A$.
In particular, we will be interested in the difference sequence $\Delta S_{\theta}$ where the $n$-th term is denoted by
 $\delta_{\theta} (n) = s_{n+1} - s_{n}$.
The differences $\delta_{\theta} (n)$ are generally decreasing as $n$ increases, but not strictly. Moreover, because $\Delta S_{\theta}$ contains repetitions, we may assign each element of $\Delta S_{\theta}$ to an integer (an element in $\mathscr{A}$) through a bijective mapping $\lambda$: Let $\lambda( \delta_{\theta} (0)) = 0$. If $n_{1}$ is the least value such that $\delta_{\theta} (n_{1}) \ne \delta_{\theta} (0)$, then $\lambda(\delta_{\theta} (n_{1}))=1$.
Then $\lambda(\delta_{\theta} (n_{2}))=2$, where $n_{2}$ is the least value such that $\delta_{\theta} (n_{2})$ is unequal to either $\delta_{\theta} (n_{1})$ or $\delta_{\theta} (n_{0})$, and so on. More than a simple coding of differences, the mapping $\lambda$ contains a deeper significance through a connection with the continued fraction expansion of $\theta$, to be discussed in \S \ref{morecf}.

Finally,
$\Lambda_{\theta} = (\lambda(\delta_{\theta} (0)),\ \lambda(\delta_{\theta} (1)),\ \lambda(\delta_{\theta} (2)),\ \lambda(\delta_{\theta} (3)), \ldots)$.
That is, $\Lambda_{\theta}$ is the word obtained from the sequence of differences in $\Delta S_{\theta}$, encoded by $\lambda$.
We refer to $\Lambda_{\theta}$ as the {\it Lambda word generated by} $\theta$.
 
\section{Examples of Lambda words}\label{exsec}
Let $\vartheta = \log_{2}3$.
The table below presents $\Lambda_{\vartheta}$ (\seqnum{A216448}) as the encoded differences of the ascending sequence $S_{\vartheta}$. 
\[
\begin{array}{c|cccccccc}
n&0&1&2&3&4&5&6&7\\
\hline
S_{\vartheta} & 0+0\vartheta &1+0\vartheta& 0+1\vartheta& 2+0\vartheta& 1+1\vartheta& 3+0\vartheta& 0+2\vartheta& 2+1\vartheta\\
\Delta S_{\vartheta} & 1-0\vartheta& -1+1\vartheta& 2-1\vartheta& -1+1\vartheta& 2-1\vartheta& -3+2\vartheta& 2-1\vartheta &\\
\Lambda_{\vartheta} & 0&1&2&1&2&3&2&
\end{array}
\]

Because $1-0\vartheta$ is the ``$0$-th'' difference in $\Delta S_{\vartheta}$, it maps, under $\lambda$, to 0. When $n$ is 1 or 3, $\delta_{\vartheta} (n) = -1+1\vartheta$. Then $\lambda(\delta_{\vartheta} (1)) = \lambda(\delta_{\vartheta} (3)) = 1$. Similarly, $\lambda(\delta_{\vartheta} (2)) = \lambda(\delta_{\vartheta} (4)) = \lambda(\delta_{\vartheta} (6)) = 2$,
and $\lambda(\delta_{\vartheta} (5)) = 3$, and so on.

According to Kimberling \cite{kim97}, the sequence of integers $i_{n}$ from $S_{\theta}$ is the {\it signature} of $\theta$, whereas $j_{n}$ gives the signature of $1/\theta$. Kimberling also introduces the {\it interspersion array}.
Figure \ref{rank} shows the interspersion array associated with the signature sequence of $1/\vartheta^{-1}$. We modify Kimberling's definition of the array so as to include zero \cite[p.\ 313]{kim93}: An array $A = (a_{ij})$, $i \ge 0, j \ge 0$, of non-negative integers is an {\it interspersion} if

\begin{enumerate}
\item the rows of $A$ comprise a partition of the non-negative integers;
\item every row of $A$ is an increasing sequence;
\item every column of $A$ is an increasing (possibly finite) sequence;
\item if $(u_{j})$ and $(v_{j})$ are distinct rows of $A$ and if $p$ and $q$ are any indices for which $u_{p} < v_{q} < u_{p+1}$, then $u_{p+1} < v_{q+1} < u_{p+2}$.
\end{enumerate}

According to Kimberling and Brown \cite{kim04}, an array such as the one shown in Figure \ref{rank} is {\it transposable}, because substituting $1/\theta$ for $\theta$ yields the transpose of the array. Adding one to each element in column one produces \seqnum{A022330}, and the same addition on elements of the first row gives \seqnum{A022331}. The relationship between transposable interspersions and Lambda words is this: Connecting the elements of the Figure \ref{rank} array in ascending order determines a sequence of vectors that begins $(1,0),\ (-1,1),\ (2,-1),\ (-1,1)$, and these are also the coefficients for the successive elements of the difference sequence $\Delta S_{\vartheta}$. Labeling each distinct vector starting with 0 yields $\Lambda_{\vartheta}=(0,1,2,1,\ldots)$. (See \seqnum{A167267} for a transposable interspersion of the signature sequence of $\phi$.)

\begin{figure}[h]
\[
\begin{array}{cccccccccccc}
 0 & 1 & 3 & 5 & 8 & 12 & 16 & 21 & 27 & 33 & 40 & 47 \\
 2 & 4 & 7 & 10 & 14 & 19 & 24 & 30 & 37 & 44 & 52 & \text{} \\
 6 & 9 & 13 & 17 & 22 & 28 & 34 & 41 & 49 & \text{} & \text{} & \text{} \\
 11 & 15 & 20 & 25 & 31 & 38 & 45 & 53 & \text{} & \text{} & \text{} & \text{} \\
 18 & 23 & 29 & 35 & 42 & 50 & \text{} & \text{} & \text{} & \text{} & \text{} & \text{} \\
 26 & 32 & 39 & 46 & 54 & \text{} & \text{} & \text{} & \text{} & \text{} & \text{} & \text{} \\
 36 & 43 & 51 & \text{} & \text{} & \text{} & \text{} & \text{} & \text{} & \text{} & \text{} & \text{} \\
 48 & \text{} & \text{} & \text{} & \text{} & \text{} & \text{} & \text{} & \text{} & \text{} & \text{} & \text{} \\
\end{array}
\]
\[
{\begin{array}{|c|c|c|c|c|c|c|c|}
\hline
0 \rightarrow 1 & 1 \rightarrow 2 & 2 \rightarrow 3 & 3 \rightarrow 4 &4 \rightarrow 5 &5 \rightarrow 6 &6 \rightarrow 7 &7 \rightarrow 8 \\
 (1,0) &(-1,1) &(2,-1) &(-1,1)&(2,-1)&(-3,2)&(2,-1)&(2,-1) \\
 0 & 1 & 2 & 1 & 2 & 3 & 2 & 2 \\
 \hline
\end{array} }
\]
\caption{Interspersion array for sequence $i(\vartheta^{-1})$ and its path of vectors.}
\label{rank}
\end{figure}

\begin{figure}[h]
\begin{center}
\includegraphics[scale = 1.11, angle = 0]{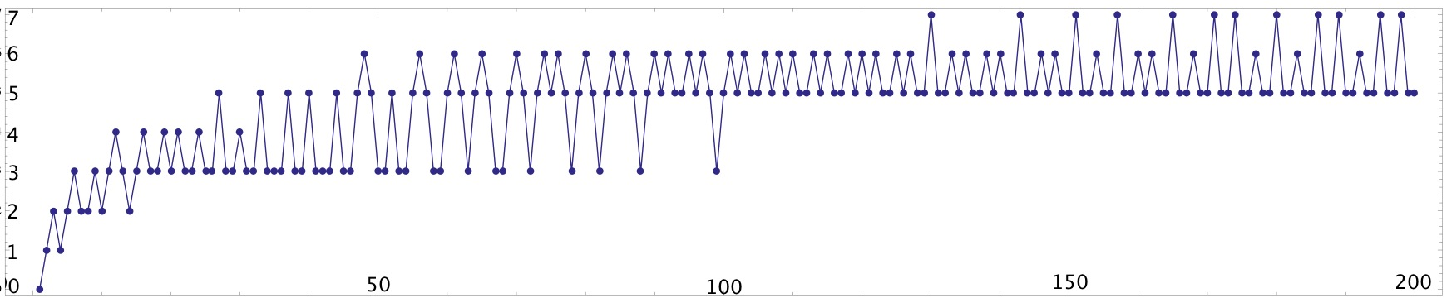}
\end{center}
$\Lambda_{\vartheta}=$
0, 1, 2, 1, 2, 3, 2, 2, 3, 2, 3, 4, 3, 2, 3, 4, 3, 3, 4, 3, 4, 3, 3, 4, 3, 3, 5, 3, 3, 4, 3, 3,
5, 3, 3, 3, 5, 3, 3, 5, 3, 3, 3, 5, 3, 3, 5, 6, 5, 3, 3, 5, 3, 3, 5, 6, 5, 3, 3, 5, 6, 5, 3, 5, 6, 5, 3, 3,
5, 6, 5, 3, 5, 6, 5, 6, 5, 3, 5, 6, 5, 3, 5, 6, 5, 6, 5, 3, 5, 6, 5, 6, 5, 5, 6, 5, 6, 5, 3, 5, 6, 5, 6, 5,
5, 6, 5, 6, 5, 6, 5, 5, 6, 5, 6, 5, 5, 6, 5, 6, 5, 6, 5, 5, 6, 5, 6, 5, 5, 7, 5, 5, 6, 5, 6, 5, 5, 6, 5, 6,
5, 5, 7, 5, 5, 6, 5, 6, 5, 5, 7, 5, 5, 6, 5, 5, 7, 5, 5, 6, 5, 6, 5, 5, 7, 5, 5, 6, 5, 5, 7, 5, 5, 7, 5, 5,
6, 5, 5, 7, 5, 5, 6, 5, 5, 7, 5, 5, 7, 5, 5, 6, 5, 5, 7, 5, 5, 7, 5, 5$,\ldots$
\caption{Graph of first 200 elements of $\Lambda_{\vartheta}$.}
\label{biggraph}
\end{figure}

Figure \ref{biggraph} provides a graphical representation of the first 200 elements of $\Lambda_{\vartheta}$. It shows that $\Lambda_{\vartheta}$ gradually increases with $n$, but with repetitions and switchbacks, allowing for an abundance of palindromes. Figure \ref{biggraph} suggests other properties of Lambda words.
For example, each integer appears a finite number of times: After smaller integers die out, larger ones take their place. 

Of particular interest is \seqnum{A216763}, the Fibonacci Lambda word, $\Lambda_{\phi}$, where $\phi = \frac{\sqrt{5}+1}{2}$. Figure \ref{fibograph} presents its first 196 elements. Because of the slow convergence of the continued fraction of $\phi$, successive members of $\Lambda_{\phi}$ differ by at most unity.

\begin{figure}[h]
\begin{center}
\includegraphics[scale = 1.15, angle = 0]{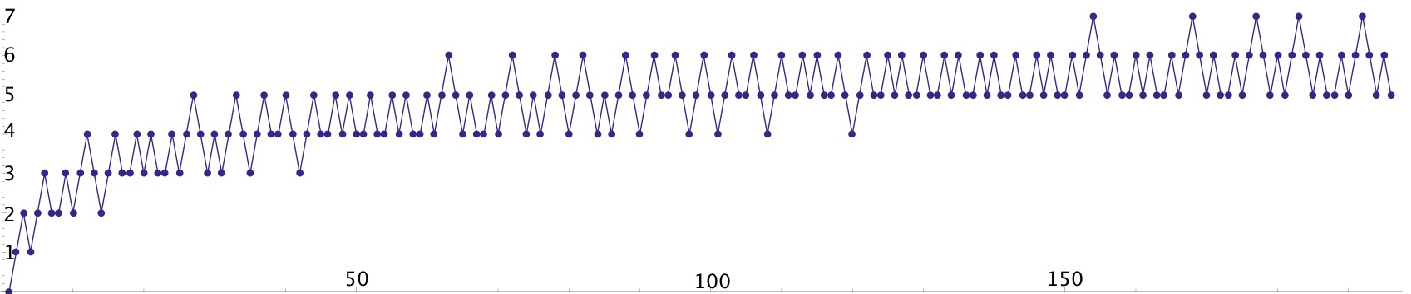}
\end{center}

$\Lambda_{\phi}=$
0, 1, 2, 1, 2, 3, 2, 2, 3, 2, 3, 4, 3, 2, 3, 4, 3, 3, 4, 3, 4, 3, 3, 4, 3, 4, 5, 4, 3, 4, 3, 4, 5, 4, 3, 4,
 5, 4, 4, 5, 4, 3, 4, 5, 4, 4, 5, 4, 5, 4, 4, 5, 4, 4, 5, 4, 5, 4, 4, 5, 4, 5, 6, 5, 4, 5, 4, 4, 5, 
4, 5, 6, 5, 4, 5, 4, 5, 6, 5, 4, 5, 6, 5, 4, 5, 4, 5, 6, 5, 4, 5, 6, 5, 5, 6, 5, 4, 5, 6, 5, 4, 5,
 6, 5, 5, 6, 5, 4, 5, 6, 5, 5, 6, 5, 6, 5, 5, 6, 5, 4, 5, 6, 5, 5, 6, 5, 6, 5, 5, 6, 5, 5, 6, 5, 6, 
5, 5, 6, 5, 6, 5, 5, 6, 5, 5, 6, 5, 6, 5, 5, 6, 5, 6, 7, 6, 5, 6, 5, 5, 6, 5, 6, 5, 5, 6, 5, 6, 7,
 6, 5, 6, 5, 5, 6, 5, 6, 7, 6, 5, 6, 5, 6, 7, 6, 5, 6, 5, 5, 6, 5, 6, 7, 6, 5, 6, 5, \ldots

\caption{Graph of first 196 elements of $\Lambda_{\phi}$.\label{fibograph}}
\end{figure}

The Lambda word generated by $\pi -2$, \seqnum{A216764}, presents a very different profile, as shown in Figure \ref{pigraph}. The continued fraction expansion, $[1, 7, 15, 1, 292, \ldots ]$, converges relatively rapidly. Note that the greatest difference between successive values so far is 7, which follows from the value of the second partial quotient of the expansion. Further on, the word will exhibit differences of 15, 292, and so on.
 
\begin{figure}[h]
\begin{center}
\includegraphics[scale = 1.15, angle = 0]{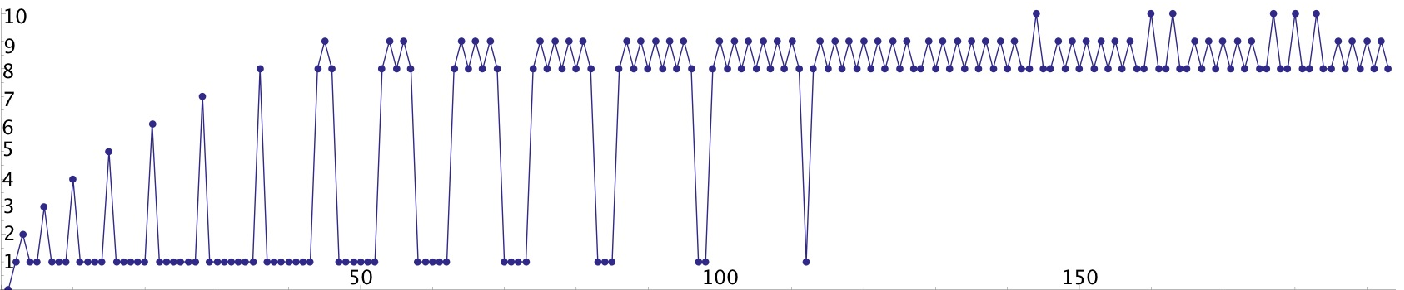}
\end{center}

$\Lambda_{\pi -2}=$
0, 1, 2, 1, 1, 3, 1, 1, 1, 4, 1, 1, 1, 1, 5, 1, 1, 1, 1, 1, 6, 1, 1, 1, 1, 1, 1, 7, 1, 1, 1, 1, 
1, 1, 1, 8, 1, 1, 1, 1, 1, 1, 1, 8, 9, 8, 1, 1, 1, 1, 1, 1, 8, 9, 8, 9, 8, 1, 1, 1, 1, 1, 8, 9, 8, 
9, 8, 9, 8, 1, 1, 1, 1, 8, 9, 8, 9, 8, 9, 8, 9, 8, 1, 1, 1, 8, 9, 8, 9, 8, 9, 8, 9, 8, 9, 8, 1, 1, 
8, 9, 8, 9, 8, 9, 8, 9, 8, 9, 8, 9, 8, 1, 8, 9, 8, 9, 8, 9, 8, 9, 8, 9, 8, 9, 8, 9, 8, 8, 9, 8, 9, 
8, 9, 8, 9, 8, 9, 8, 9, 8, 9, 8, 8, 10, 8, 8, 9, 8, 9, 8, 9, 8, 9, 8, 9, 8, 9, 8, 8, 10,8, 8, 10,
 8, 8, 9, 8, 9, 8, 9, 8, 9, 8, 9, 8, 8, 10, 8, 8, 10, 8, 8, 10, 8, 8, 9, 8, 9, 8, 9, 8, 9, 8, \ldots 
\caption{Graph of first 193 elements of $\Lambda_{\pi -2}$.\label{pigraph}}
\end{figure}

\section{Lambda words and continued fractions}\label{cfsec}

The structure of a Lambda word is determined by the continued fraction expansion of $\theta$, its generating value. We demonstrate some of these relationships here, beginning with a theorem about differences in $\Delta{S_{\theta}}$. Standard theorems regarding continued fractions can be found in many texts \cite{car12, jon55, old63, khi35}. If $[t_{0};t_{1},\ldots,t_{k},\ldots]$ is a continued fraction, then the $t_{k}$ are its {\it partial quotients} and $c_{0}=[t_{0}]$, $c_{1}=[t_{0};t_{1}]$, $c_{2}=[t_{0};t_{1},t_{2}]$, etc., are its {\it principal convergents}. There are two formal convergents for any continued fraction: $c_{-2} = 0/1$ and $c_{-1} = 1/0$. If $c_{k} = a_{k}/b_{k}$, then $\gcd(a_{k},b_{k}) = 1$. Moreover, $c_{k+1} = (t_{k+1}a_{k}+a_{k-1})/(t_{k+1}b_{k}+b_{k-1})$. If $t_{k} > 1$ we define integer $t$ such that $1 \le t < t_{k}$. Then $[t_{0};t_{1},\ldots,t_{k-1},t]$ is an {\it intermediate convergent} of the continued fraction.
In this paper, the term ``convergent'' means a principal convergent {\it or} an intermediate convergent.
Khinchin \cite{khi35} devises the term ``best approximation of the second kind'' for a rational $p/q$ of some real number $\theta$ if $|q\theta-p| < |q^{\prime}\theta - p^{\prime}|$ whenever $q \ge q^{\prime}$. (A best approximation of the first kind involves the value $|\theta - p/q|$ and plays no role here. A ``best approximation'' will refer to one of the second kind.)

In addition to the best ``two-sided'' approximations described, Richards \cite{ric81} further defines a best one-sided approximation: Take $p/q < \theta$. If $0 < q\theta-p < q^{\prime}\theta-p^{\prime}$ whenever $q \ge q^{\prime}$, then $p/q$ is a best lower, or left, approximation. For any best left approximation it follows that $p = \lfloor q\theta \rfloor$. When $p/q$ is a best higher, or right approximation, $\theta <p/q$, and $q = \lfloor p/\theta \rfloor$.
With $1 < \theta < 2$, any intermediate convergent is a best one-sided approximation and (with a single exception)
any principal convergent is a best two-sided approximation. Certainly, any best two-sided approximation is the best on its side as well. (The exception is in the case of the formal convergent $0/1$. It is not a best left approximation because the convergent $1/1$ is also less than $\theta$ and has the same denominator.)
\begin{theorem}\label{iffconv}
$| A-B\theta| \in \Delta S_{\theta} \iff A/B$ is a best approximation \emph{(}of the second kind\emph{)} of $\theta$.
\end{theorem}

\begin{proof}(See also \cite[Theorem 1]{car12}.)

$\Rightarrow$. 
We prove the case $\delta_{\theta} (n) = B\theta - A$, i.e., $A < B\theta$.
\[
\delta_{\theta} (n) = {s_{n+1}} - {s_{n}} = (i_{n+1}+j_{n+1}\theta) -(i_{n}+j_{n}\theta) = B\theta -A.
\]

Then $(j_{n+1}- j_{n})\theta = B\theta$ and $ i_{n} - i_{n+1} = A$, and so ${s_{n+1}\ge B\theta }$ and ${s_{n}} \ge A$.
Define $g$ such that
\begin{equation}\label{g}
g = ({s_{n+1}}-B\theta) = ({s_{n}} - A).
\end{equation}

Then $g \in S_{\theta}$. 
Assume $B\theta-A >0$ but $A/B$ is not a best left approximation. Then there exist successive best left approximations $p/q$ and $p^{\prime}/q^{\prime}$ with $q \le B < q^{\prime}$ such that $q\theta -p < B\theta - A$. 

Then
\[
0< q\theta -p <B\theta - A.
\]

Adding $A+g$ we obtain
\[
A +g < A+g-p+q\theta < B\theta +g.
\]
Clearly, $A+g-p+q\theta = s_{n^{\prime}} \in S_{\theta}$, and so, by (\ref{g}), ${s_{n}} < s_{n^{\prime}} < {s_{n+1}}$. Contradiction. (For the case $\delta_{\theta} (n) = A - B\theta$, redefine $g$ as $\hat{g} = ({s_{n+1}} - A) = ({s_{n}} - B\theta)$.)

$\Leftarrow$. The second part of the proof is again by contradiction: Assume $A/B$ is a best left approximation.
Now assume that $A = s_{n}$ and $B\theta = s_{n+k}$ with $k>1$.
Then we must have $s_{n+1} = i+j\theta$ such that
\begin{equation}\label{ij}
 A < i+j\theta <B\theta
\end{equation}
Then $j < B$. Because $\lfloor B\theta \rfloor = A$, $i < A$.

Subtracting $A$ from (\ref{ij}) we obtain
\begin{equation}\label{ij2}
 0 < j\theta -(A-i) <B\theta-A.
\end{equation}
 Because $A/B$ is a best left approximation, there is no integer $k$ such that $j\theta - k < B\theta - A$. In particular, $j\theta - (A-i) \not< B\theta -A$. contradicting (\ref{ij2}). Then if $A/B$ is a best left approximation, $\delta_{\theta}(n) = (B\theta - A)$, therefore $(B\theta - A) \in \Delta S_{\theta}$.
 \end{proof}

\subsection{The mapping $\lambda$ formalized}
\label{morecf}

The continued fraction allows for an algorithmic form of the mapping $\lambda$: In the case of the formal convergent $1/0$, define $\lambda(1-0\theta) = \lambda(\delta_{\theta} (0)) = 0$. Otherwise, for each distinct difference $|A-B\theta|$, express $A/B$ as a finite continued fraction $[1;t_{1}, t_{2}, \ldots, t_{k}]$. Then $\lambda$ may be computed as
\[
 \lambda |A-B\theta| = \sum_{i = 0}^{k} t_{i}.
\]

\begin{figure}[h]
\[
\begin{array}{| c | c | l | c |}
|A-B\vartheta| &A/B & [1;t_{1}, \ldots, t_{k}] & \sum_{i = 0}^{k} t_{i}\\
\hline
1 & 1/0 & & 0\\
\vartheta - 1 &1/1 & [1] & 1\\
2 - \vartheta &2/1 & [1,1 ] & 2\\
2\vartheta -3 & 3/2 & [1,1,1 ] & 3\\
5 - 3\vartheta & 5/3 & [1,1,1,1 ] & 4\\
8 - 5\vartheta & 8/5 & [1,1,1 ,2] & 5\\
7\vartheta - 11 &11/7 & [1,1,1,2,1 ] & 6\\
12\vartheta -19 &19/12 & [1,1,1,2,2 ] & 7\\
27 - 17\vartheta &27/17 & [1,1,1,2,2,1 ] & 8\\
46 - 29\vartheta &46/29 & [1,1,1,2,2,2 ] & 9\\
65 - 41\vartheta &65/41 & [1,1,1,2,2,3 ] & 10\\
\end{array}
\]
\caption{Computing $\lambda |A-B\vartheta|$ by sums of partial quotients.\label{sumspq}}
\end{figure}

Figure \ref{sumspq} shows the first eleven convergents of the continued fraction expansion of $\vartheta = \log_{2}3$ under $\lambda$. Every non-negative integer is equal to exactly one of the summations, which is guaranteed for all $\theta$ by the second half of the proof of Theorem \ref{iffconv}. As a result, the mapping $\lambda$ associates the difference $|A-B\theta|$ with the row of the Stern-Brocot tree on which $A/B$ is found \cite[pp.\ 116--117]{ste58, gra94}.

\subsection{The Hurwitz chain}

In what follows we will make use of an oft-noted connection between Farey series, continued fractions, and mediants \cite{gol88, ric81}. If $a/b < c/d$ are consecutive fractions in a Farey series then $bc - ad = 1 = \gcd(a,b) = \gcd(c,d)$. Then we say that $(a/b$, $c/d)$ form a {\it Farey pair}. Whenever $a/b < \theta <c/d$, the Farey pair is also a pair of best left-right approximations of $\theta$. The {\it Hurwitz chain} for $\theta$ contains all such pairs $(a/b,c/d)$ \cite{hur94}. If $(a/b,c/d)$ belongs to the Hurwitz chain for $\theta$, then, if $a/b$ is an intermediate convergent, $c/d$ is the previous principal convergent and conversely. Otherwise, $a/b$ and $c/d$ are consecutive principal convergents. The mediant of these fractions, $(a+c)/(b+d)$ falls between them and is also a convergent. This implies that $(a+c)/(b+d)$ forms a member of the Hurwitz chain for $\theta$ with either $a/b$ or $c/d$. 

Let $A/B$ and $p/q$ be a pair of best left-right approximations of $\theta$ such that $|A-B\theta| > |p-q\theta |$. Then either $(A/B,p/q)$ or $(p/q,A/B)$ belongs to the Hurwitz chain for $\theta$. If $A/B$ is a principal convergent, then $p/q$ is the principal convergent that immediately follows $A/B$, and if $A/B$ is an intermediate convergent, then $p/q$ is the principal convergent that immediately precedes $A/B$.
Furthermore, $q > B$ when $A/B$ is a principal convergent, $q < B$ otherwise. Either way, $(A+p)/(B+q)$ is the first convergent that succeeds both. 

\section{Counting letters}\label{countsec}
A recurrent factor of an infinite word is a factor that appears infinitely often. Here we show that there are no recurrent letters, consequently, no recurrent factors in a Lambda word. We let $\lambda|A-B\theta| = x$ and let ${|\Lambda_{\theta}|}_{x}$ represent the number of times the letter $x$ occurs in $\Lambda_{\theta}$. 

\begin{theorem}\label{howmany}
${|\Lambda_{\theta}|}_{x} = pq$.
\end{theorem}

\begin{proof}
We prove the theorem for $A/B <\theta$; the proof for $A/B > \theta$ is left as an exercise.
First, let us define
\[
\begin{array}{lcccc}
X &= &B\theta&-&A\\
Y &= &p&-&q\theta\\
Z &= &(B+q)\theta&-&(A+p)
\end{array}
\]
Then $Z = X-Y$. It follows that $A/B < (A+p)/(B+q) < \theta < p/q$. Furthermore, $Y<X$ and $Z<X$.
We define a subset ${S(\theta)}_{x}$ of ${S(\theta)}$ such that
\[
{S(\theta)}_{x} = \{s_{n}\mid \delta_{\theta} (n) = X \text{ where } \lambda(X) = x\}. 
\]
That is, $s_{n} \in {S(\theta)}_{x} \iff s_{n}+X = s_{n+1}$. We will show that there are $pq$ elements in ${S(\theta)}_{x}$, and so ${|\Lambda_{\theta}|}_{x} = pq$. In order to do so, we prove that 
\begin{equation}\label{sxset}
{S(\theta)}_{x} = \{a+b\theta \mid A \le a < (A+p); 0 \le b <q\}
\end{equation}

Because $s_{n} \in {S(\theta)}$ it immediately follows that $0 \le b$.

Further, $s_{n} + X \in S(\theta)$ implies $a+b\theta + (B\theta -A) \in S(\theta)$ and so $A \le a$.

If $s_{n} \in {S(\theta)}_{x}$ and $b-q \ge 0$, then it would follow that $s_{n} + Y = a+b\theta + (p - q\theta) \in S(\theta)$. But $Y<X$ and so $s_{n} < s_{n} + Y < s_{n} + X$, that is, $s_{n+1} \ne s_{n} + X$. Thus, $s_{n} \notin {S(\theta)}_{x}$, contrary to hypothesis. Therefore $b < q$. 

If $a \ge A+p$, then $a+b\theta + ((B+q)\theta - (A+p)) = s_{n} + Z \in S(\theta)$. 
But $Z < X$ and so $s_{n} < s_{n} + Z < s_{n} + X$, contradicting $s_{n+1} = s_{n} + X$. Then $a < A+p$.

Therefore $A \le a < (A+p)$ and $0 \le b < q$ and so there are $pq$ elements in ${S(\theta)}_{x}$, i.e., ${|\Lambda_{\theta}|}_{x} = pq$.
\end{proof}

We note that, when $A/B > \theta$, we have
\begin{equation}\label{sxset2}
{S(\theta)}_{x} = \{a+b\theta \mid 0 \le a < p; B \le b < B+q\}.
\end{equation}

\section{Two-Letter factors of Lambda words}\label{langsec}

To determine the limit on the number of different letters in palindromes in a Lambda word, we explore the two-letter factors of the word, which we define as members of the set $\mathcal{F}_{2}(\Lambda_{\theta})$. For $x,y \in \mathscr{A}$, we determine the conditions under which $xy \in \mathcal{F}_{2}(\Lambda_{\theta})$.
 
\subsection{Letter repetitions}

To begin, we investigate the two-letter factors of the form $xx = x^{2}$.
Let $t$ be the greatest integer such that $x^{t}$ is a factor of $\Lambda_{\theta}$.
Let $t_{k}$ be the $k$-th partial quotient of the continued fraction expansion of $\theta = [1;t_{1},\ldots,t_{k},\ldots]$.
\begin{theorem}\label{power}
If $x = \lambda|A-B\theta|$ where $A/B$ is an intermediate convergent, $xx \notin \mathcal{F}_{2}(\Lambda_{\theta})$.  If $A/B$ is the principal convergent $a_{k}/b_{k}$, $xx \in \mathcal{F}_{2}(\Lambda_{\theta})$. There are two special cases when $A/B$ is $a_{-1}/b_{-1}$ or $a_{0}/b_{0}$.
\end{theorem}
\begin{proof}
Let $|A-B\theta| = X$. If a $xx$ is a factor of $\Lambda_{\theta}$, then ${S(\theta)}_{x}$ would have to include both $a+b\theta$ and $a+b\theta + X$. If $x^{t}$ is a factor, then $a+b\theta + (t-1)X$ must be an element of ${S(\theta)}_{x}$. We can solve for $t$ by invoking the conditions on $a$ and $b$ in ${S(\theta)}_{x}$, as shown in (\ref{sxset}) and (\ref{sxset2}).
\[
\left.
\begin{array}{c}
A \le a + (t-1)A < A+p \\
0 \le b + (t-1)B < q \\
\end{array}\right\} \text{ by (\ref{sxset})}
\]
\[
\left.
\begin{array}{c}
0 \le a + (t-1)A < p \\
B \le b + (t-1)B < B+q \\
\end{array}\right\} \text{ by (\ref{sxset2})}
\]
We are interested in the maximal value for $t$, so we substitute minimal values for $a$ and $b$, which simplifies to two results
\begin{equation}\label{t}
t < 1+ \dfrac{p}{A} \text{ and } t < 1 + \dfrac{q}{B}.
\end{equation}

When $A/B$ is an intermediate convergent, $p/q$ is the previous principal convergent, and so $B>q$ (and $A>p$). Therefore $t < 1 + (q/B) < 2$, i.e., $t = 1$. Therefore, when $A/B$ is an intermediate convergent $xx\notin \mathcal{F}_{2}(\Lambda_{\theta})$.

If $A/B$ is a principal convergent, then it is associated with the partial quotient $t_{k}$, and $p/q$ is the the next convergent, associated with $t_{k+1}$. Let $p^{\prime}/q^{\prime}$ represent the principal convergent that immediately precedes $A/B$. Then $q/B = t_{k+1}+(q^{\prime}/B)$. Whenever $k>0$, $q^{\prime}/B \le 1$, and so we can solve for $t$: $t< 1+ t_{k+1} +(q^{\prime}/B)$. Then $t = 1+t_{k+1}$. (Alternatively, $t< 1+ t_{k+1}+(p^{\prime}/A)$ yields the same result for $k>0$.) Because $t_{k+1} \ge 1$, then $t>1$ and so whenever $A/B$ is a principal convergent, $xx\in \mathcal{F}_{2}(\Lambda_{\theta})$.

There are two special cases that involve principal convergents when $k= -1$ or $0$. In both cases, $t = t_{k+1}$: For $k = -1$, $A/B = 1/0$, $p/q = 1/1$. The first inequality of (\ref{t}) resolves to $t < 2$, but the second entails division by zero and is rejected. Then $t = t_{k+1} = t_{0} = 1$ and $xx\notin \mathcal{F}_{2}(\Lambda_{\theta})$. For $k = 0$, $A/B = 1/1$, $p/q = (1 + t_{1})/t_{1}$ The inequalities in (\ref{t}) resolve to $t < 2 + t_{1}$ and $t < 1+t_{1}$. Both must be satisfied but the second is more restrictive, so again $t = t_{1}=t_{k+1}$ Then $xx \in \mathcal{F}_{2}(\Lambda_{\theta})$ whenever $t_{1} \ge 2$.
\end{proof}

\subsection{Square-free two-letter factors}
\label{difl}

For $u,v, \in \mathscr{A}$ $u \ne v$, we examine conditions under which $uv \in \mathcal{F}_{2}(\Lambda_{\theta})$. We let $r$ be a word over $\mathscr{A}$ such that $r \in \mathcal{F}(\Lambda_{\theta})$.
We begin by examining cases of pairs of letters $(u,v)$ with $u\prec v$ such that $vru \notin \mathcal{F}(\Lambda_{\theta})$.

Recall that $X-Y = Z$. We define $W = X+Y$. We write $a/b \oplus c/d = (a+c)/(b+d)$ to indicate that the fraction on the right side is the mediant of the two on the left. Similarly, $a/b \ominus c/d = (a-c)/(b-d)$.
We define a function $f$ on the difference $|A-B\theta|$ such that $f|A-B\theta| = A/B$. 
The pair $(f(X),f(Y))$ is a member of the Hurwitz chain for $\theta$. 
Then $f(X)\oplus f(Y) = f(Z)$ and so $(f(Y),f(Z))$ is also a member of the Hurwitz chain. It follows that $f(X)\ominus f(Y) = f(W)$ and so either $(f(W),f(X))$ or $(f(W),f(X))$ is a member of the Hurwitz chain as well. Certainly $f(W),\ f(X),\ f(Y)$, and $f(Z)$ are all convergents to $\theta$,
and so by Theorem \ref{iffconv}, $W$, $X$, $Y$, and $Z$, are differences in $\Delta S_{\theta}$.
See Figure \ref{WandfW}.

\begin{figure}[h]
\[
\begin{array}{|ccccccc|}
\hline
\multicolumn{7}{|c|}{\text{\underline{Differences}}}\\
W &=& X+Y &= &\mid(B-q)\theta&-&(A-p)\mid \\
X &= &&&\mid B\theta&-&A \mid \\
Y &= &&&\mid p&-&q\theta \mid \\
Z &= &X-Y &= &\mid (B+q)\theta&-&(A+p)\mid \\
\hline
\multicolumn{7}{|c|}{\text{\underline{Ratios}}}\\
f(W) &=& f(X)\ominus f(Y) &= &\multicolumn{3}{c|}{\mid A-p\mid / \mid B-q\mid} \\
f(X) &= &&&\multicolumn{3}{c|}{A/B}\\
f(Y) &= &&&\multicolumn{3}{c|}{p/q} \\
f(Z) &= &f(X)\oplus f(Y) &= &\multicolumn{3}{c|}{(A+p)/(B+q)}\\
\hline
\multicolumn{7}{|c|}{\text{\underline{Letters}}}\\
\lambda(W) &=& w &&&&\\
\lambda(X) &=& x &&&&\\
\lambda(Y) &=& y &&&&\\
\lambda(Z) &=& z &&&&\\
\hline
\end{array}
\]

\caption{Differences, ratios, and letters that govern potential two-letter factors.}
\label{WandfW}
\end{figure}

The following theorem provides groundwork for Theorem \ref{ishurw}, which will show that $uv$ is a factor if and only if $f((U),f(V))$ belongs to the Hurwitz chain for $\theta$. As before, for difference $U$, we let $\lambda(U) = u$. 

Let $r$ be a (possibly empty) factor of $\Lambda_{\theta}$.
\begin{theorem}\label{notzw}
If $(f(X),f(Y))$ belongs to the Hurwitz chain for $\theta$, $zrw \notin \mathcal{F}(\Lambda_{\theta})$. Furthermore, $wz \notin \mathcal{F}_{2}(\Lambda_{\theta})$.
\end{theorem}

\begin{proof}

We prove the first part of the theorem by showing that the difference ${W}$ never comes after $Z$ in $\Delta S_{\theta}$. We prove the case $X = A-B\theta$:
\[
\begin{array}{ccccc}
W &=& X&+&Y \\ 
X &=& A&-&B\theta \\
Y &=& q\theta &- &p \\
Z &=& (A+p) &- &(B+q)\theta
\end{array}
\]

The smallest member of $S_{\theta}$ for which $\delta_{\theta} (n) = Z$ is $(B+q)\theta$.
Therefore, if $W$ appears after $Z$, then there exists $\delta_{\theta} (m) = W$ such that $s_{m}\ge A+p > (B+q)\theta$. Consequently, $s_{m} > (B+q)\theta - Y = p+B\theta$.  But $s_{m} < s_{m} +Y < s_{m} +X < s_{m} +W$ and at least one of these is true: $s_{m} +Y \in S_{\theta}$ and/or $s_{m} +X \in S_{\theta}$. Then 
$W \ne \delta_{\theta} (m)$ whenever $s_{m} \ge (B+q)\theta$.
Therefore $zrw \notin \mathcal{F}(\Lambda_{\theta})$.

It is also true that $wz$ is not a factor. The first part of the proof has shown that $\delta_{\theta} (n) \ne W$ when $s_{n} > (B+q)\theta$ and so we need only show that if $s_{m} = (B+q)\theta$, then $\delta_{\theta}({m-1}) \ne W$. Clearly, $(B+q)\theta - Y \in S_{\theta}$. However, $(B+q)\theta - W < (B+q)\theta - Y < (B+q)\theta = s_{m}$ and so $W \ne \delta_{\theta}({m-1})$.
\end{proof}

Certainly if $z^{\prime} \succeq z$ then $z^{\prime}rw$ is not a factor of $\Lambda_{\theta}$.

\vskip .1 in 

Theorem \ref{power} revealed the conditions under which $x^{2} \in \mathcal{F}_{2}(\Lambda_{\theta})$. The following theorem demonstrates the conditions under which $xy \in \mathcal{F}_{2}(\Lambda_{\theta})$ is comprised of different letters.
Let $u=\lambda(U)$, $v=\lambda(V)$, and $u\ne v$.
\begin{theorem}\label{ishurw}
$uv \in \mathcal{F}_{2} (\Lambda_{\theta}) \iff (f(U),f(V))$ belongs to the Hurwitz chain for $\theta$. 
\end{theorem}

\begin{proof}
Theorem \ref{notzw} has limited membership in $\mathcal{F}_{2}(\Lambda_{\theta})$ for a fixed $\{x,y\}$ to a pair of letters in $\{w,x,y,z\}$.
Then $U,V \in \{W,X,Y,Z\}$.
Let $U = |A^{\ast}-B^{\ast}\theta|$, $V = |p^{\ast}-q^{\ast}\theta|$. By Theorem \ref{iffconv}, $A^{\ast}/B^{\ast}$ and $p^{\ast}/q^{\ast}$ are convergents. 

$\Rightarrow$ Assume $uv \in \mathcal{F}_{2}(\Lambda_{\theta})$. If $(f(U),f(V))$ does not belong to the Hurwitz chain then at least one of these is false:

\begin{enumerate}
\item{$|A^{\ast}q^{\ast}-B^{\ast}p^{\ast}| = 1$}
\item{$f(U)$ and $f(V)$ lie on opposite sides of $\theta$.}
\end{enumerate}

We test both conditions. Assume $|A^{\ast}q^{\ast}-B^{\ast}p^{\ast}| \ne 1$. Because $|Aq-Bp| = 1$, then $|A^{\ast}q^{\ast}-B^{\ast}p^{\ast}| = 1$ for all pairs of $\{f(W),f(X),f(Y),f(Z),\}$ except the pair $\{f(W),f(Z)\}$. But Theorem \ref{notzw} has already shown that neither $zw$ nor $wz$ belong to $\mathcal{F}_{2}(\Lambda_{\theta})$, therefore condition 1 is true, that is, $|A^{\ast}q^{\ast}-B^{\ast}p^{\ast}| = 1$.

Now assume that $f(U)$ and $f(V)$ are both greater than $\theta$. 
Let $\delta_{\theta} (n) =U$, $\delta_{\theta} (n+1) =V$, with $U> V$ (otherwise relabel $U\leftrightarrow V$). 
 If $uv \in \mathcal{F}_{2}(\Lambda_{\theta})$, then we have contiguous elements $s_{n}$, $s_{n}+U, s_{n}+U+V$.
 If $s_{n} =a+b\theta$, then $b\ge B^{\ast}+q^{\ast}$, so $s_{n}+V \in s_{\theta}$, but $s_{n} < s_{n}+V < s_{n}+U+V$, which implies that $U = V$, a contradiction. Similarly when $f(U)$ and $f(V)$ are both less than $\theta$. Therefore both conditions are true and so $uv \in \mathcal{F}_{2}(\Lambda_{\theta})$ implies that $(f(U),f(V))$ is a member of the Hurwitz chain.
 
 $\Leftarrow$
 Now assume   $(f(U),f(V))$ is a member of the Hurwitz chain for $\theta$.
We show that $uv \in \mathcal{F}_{2}(\Lambda_{\theta})$. We let $(f(U),f(V)) = (f(X),f(Y))$, remembering that $(f(X),f(Y))$ may represent {\it any} member of the Hurwitz chain for $\theta$. The limits on $S(\theta)x$ in Theorem \ref{howmany} will be utilized to show that $\max({S(\theta)}_{x}) + X \in {S(\theta)}_{y}$ and so $xy = uv \in \mathcal{F}_{2}(\Lambda_{\theta})$. Let $P/Q$ be the principal convergent that immediately follows $p/q$. Again letting $X = B\theta - A$,  $Y = p-q\theta$,
 \[ \max({S(\theta)}_{x}) + X = (A+p-1 + (q-1)\theta) + (B\theta-A) = p-1 + (B+q-1)\theta.
 \]
 Then 
 \[ p-1 + (B+q-1)\theta \in {S(\theta)})_{y}= \{a+b\theta \mid 0 \le a < P; q \le b < q +Q\}\]
 because $0 \le p-1 < P$ and $q \le (B+q-1) < q + Q$. Then $uv \in \mathcal{F}_{2}(\Lambda_{\theta}$. We can also show that $\max({S(\theta)}_{x}) - Y \in {S(\theta)}_{y}$ and so $vu$ also belongs to $\mathcal{F}_{2}(\Lambda_{\theta})$.
 \end{proof}

For a fixed $(f(X),f(Y))$, members of the Hurwitz chain for $\theta$ always include both $(f(Y),f(Z))$ and $(f(X),f(Y))$. If $f(W)$ is on the same side of $\theta$ as $f(X)$, then $(f(W),f(Y))$ belongs to the Hurwitz chain. Otherwise, $(f(W),f(X))$ does.
Then there are six two-letter square-free factors over $\{w,x,y,z\}$: $yz$, $xy$, $wx$ (or $wy$), and their reversals.

\begin{corollary}\label{3limit}
Palindromes in $\Lambda_{\theta}$ are over alphabets of at most three letters.
\end{corollary}
\begin{proof}
By Theorem \ref{notzw} it follows that if $z^{\prime} \succeq z$ then $z^{\prime}rw \notin \mathcal{F}(\Lambda_{\theta})$ and so $z^{\prime}rw\widetilde{r}z^{\prime}\notin \mathcal{F}(\Lambda_{\theta})$ and $w\widetilde{r}z^{\prime}rw\notin \mathcal{F}(\Lambda_{\theta})$. Then any palindrome in $\Lambda_{\theta}$ is over some alphabet $\mathscr{X} = \{x,y,z\} \subset \mathscr{A}$.
\end{proof}

\section{Rich words}\label{richsec}

The topic of palindromic complexity \cite{all03, dro95, dro99} concerns the question of the number of distinct palindromes contained in a word of length $k$. Words that achieve the maximal number of palindromes have been known alternately as ``rich'' words \cite{all03, buc09a, gle09b} or ``full'' words \cite{blo11, brl04}. We will use the term ``rich'' here, as it seems to be found more in current usage. It is known \cite{dro01} that a word of length $k$
contains at most $k +1$ distinct palindromes (including the empty word). 
A word that achieves this maximal number of distinct palindromes is a {\it rich} word 
\cite{gle09b}.
Alternately, a rich word contains $k$ non-empty distinct palindromes. Thus, the word ``large'' is rich, as each of its five letters is a palindrome. The word ``small'' is also rich, because ``s'', ``m'', ``a'' ``l'' and ``ll'' constitute five non-empty palindromes. On the other hand, ``edge'' is not rich as it contains only three distinct non-empty palindromes. All words of up to length three are rich.

Droubay, Justin, and Pirillo \cite{dro01} demonstrate an equivalent definition of rich words:
A word is rich if each of its prefixes terminates with a uni-occurrent palindromic suffix (ups).
For example, the relevant prefixes of the rich word ``indeed'' are ``{\it i}'', ``i{\it n}'', ``in{\it d}'', ``ind{\it e}'', ``ind{\it ee}'', ``in{\it deed}'', and the empty word. The italics capture the uni-occurrent palindromic suffixes. Because the reverse of a rich word is also rich, we can also say that a word is rich if each suffix contains a uni-occurrent palindromic prefix.

It follows that a word is rich if all of its factors are rich, which extends the concept to include infinite rich words. A right-infinite word is rich if each prefix of the word terminates in a ups. Recent studies of rich words have examined Sturmian \cite{luc08}, episturmian \cite{gle09a}, and other words on finite alphabets. The results here extend the topic to the Lambda word, a word defined over an an infinite alphabet. By way of motivation, we again examine $\Lambda_{\vartheta}$ where ${\vartheta} = \log_{2} 3$. The beginning of this sequence was shown in Figure \ref{biggraph}. As Figure \ref{tw} shows, all of its first twelve prefixes end with a ups.

\begin{figure}[h]
\[
\begin{array}{r|l}
0&\underline{0}\\
1&0,\ \underline{1}\\
2&0,\ 1,\ \underline{2}\\
3&0,\ \underline{1,\ 2,\ 1}\\
4&0,\ 1,\ \underline{2,\ 1,\ 2}\\
5&0,\ 1,\ 2,\ 1,\ 2,\ \underline{3}\\
6&0,\ 1,\ 2,\ 1,\ \underline{2,\ 3,\ 2}\\
7&0,\ 1,\ 2,\ 1,\ 2,\ 3,\ \underline{2,\ 2}\\
8&0,\ 1,\ 2,\ 1,\ 2,\ \underline{3,\ 2,\ 2,\ 3}\\
9&0,\ 1,\ 2,\ 1,\ \underline{2,\ 3,\ 2,\ 2,\ 3,\ 2}\\
10&0,\ 1,\ 2,\ 1,\ 2,\ 3,\ 2,\ 2,\ \underline{3,\ 2,\ 3}\\
11& 0,\ 1,\ 2,\ 1,\ 2,\ 3,\ 2,\ 2,\ 3,\ 2,\ 3,\ \underline{4}\\
\end{array}
\]
\caption{Uni-occurrent palindromic suffixes in the first twelve prefixes of $\Lambda_{\vartheta}$.\label{tw}} 
\end{figure}

The following proposition and four lemmas set up Theorem \ref{lamisrich}, which proves that $\Lambda_{\theta}$ is rich. 
We will show that every prefix of $\Lambda_{\theta}$ is terminated by a ups.

\subsection{Additive complements}
Two definitions are required:
If $c+c^{\prime} = h$, then $c^{\prime}$ is the {\it complement} of $c$ with respect to $h$.
A strictly increasing sequence $C = (c_{1}, \ldots , c_{n})$ where complements are defined with respect to $c_{1}+c_{n}$, is said to be a {\it sequence of additive complements} (or simply, ``sequence of complements'') if $c^{\prime} \in C$ whenever $c \in C$.
It follows that if $c_{a} < c_{b}$, then ${c_{a}}^{\prime} > {c_{b}}^{\prime}$ and also that ${c_{k}}^{\prime} = {c_{n+1-k}}$ where $1 \le k \le n$.

\begin{proposition}\label{palcomp} \emph{(Trivial)}.
 $C$ is a sequence of complements 
if{f} $\Delta C$ is a palindrome.
\end{proposition}
The proof is immediate from the definitions.

\subsection{Beatty sequences}
Setting $k \in \Nset$ allows for two values that approximate $\theta$: \[\dfrac{ \lfloor k\theta \rfloor }{k} < \theta < \dfrac{k}{ \lfloor k{(\theta}^{-1}) \rfloor }.\] 

These approximations include the ``best approximations of the second kind'' but form a larger set, comprised of the best left approximation for every denominator $k$, and the best right one for every numerator $k$.
We also have
\begin{equation}\label{klims}
\begin{array}{ccc}
\lfloor k\theta \rfloor &< &k\theta\\
\lfloor k{(\theta}^{-1})\rfloor \theta & < &k.
\end{array}
\end{equation}
Summing the values in (\ref{klims}) gives rise to two integer sequences:
\[
\begin{array}{ccl}
\vspace{1 mm}
K^{(-)} &= &\lfloor k\theta \rfloor + k\\
\vspace{1 mm}
K^{(+)} &= &\lfloor k{(\theta}^{-1}) \rfloor + k\\
\vspace{1 mm}
\end{array}
\] 
where $k$ runs through the positive integers.

For the next lemma, we utilize the notion of complementary Beatty sequences, the word ``complement'' here used in a sense distinct from above. {A Beatty sequence \cite{bea27, gar89} is an infinite sequence of integers $\lfloor k\alpha \rfloor$ where $\alpha$ is an irrational number. Complementary Beatty Sequences $\lfloor k\alpha \rfloor$ and $\lfloor k\beta \rfloor$ are formed whenever ${1}/{\alpha} +{1}/{\beta} = 1$. Any positive integer belongs to exactly one of $\lfloor k\alpha \rfloor$ or to $\lfloor k\beta \rfloor$.}

\begin{lemma}\label{beat}
$K^{(-)}$ and $K^{(+)}$ form a complementary pair of Beatty sequences. 
\end{lemma}

\begin{proof}
If $\alpha = (\theta+1)$, we get the Beatty Sequence $\lfloor k\alpha \rfloor = \lfloor k(\theta + 1)\rfloor = \lfloor k\theta +k \rfloor = \lfloor k \theta \rfloor +k = K^{(-)}$.
Solving ${1}/{(\theta+1)} +{1}/{\beta} = 1$ for $\beta$ gives $\frac{\theta+1}{\theta}$, which yields Beatty Sequence $\lfloor k\beta \rfloor =\lfloor k ( \frac{\theta+1}{\theta}) \rfloor = \lfloor \frac{k{\theta}}{\theta} + \frac{k}{\theta} \rfloor = k + \lfloor \frac{k}{\theta} \rfloor = K^{(+)}$. Consequently, every positive integer belongs either to $K^{(-)}$ or to $K^{(+)}$.
\end{proof}

\subsection{Nuclear sequences and palindromes}
We utilize the limits shown in (\ref{klims}) to form two classes of contiguous subsequences of $S_{\theta}$, indexed by $K^{(-)}$ or $K^{(+)}$:
\[
\begin{array}{l}
N_{K^{(-)}} = (s_{n} \mid \lfloor k\theta \rfloor \le s_{n} \le k \theta)\\
N_{K^{(+)}} =(s_{n} \mid \lfloor k{(\theta}^{-1})\rfloor \theta \le s_{n} \le k )
\end{array}
\]

Letting $K = K^{(-)} \cup K^{(+)}$, then $N_{K}$ refers to a member of either of these types of ``nuclear'' sequences. (It is not difficult to show that the nuclear sequences cover $S_{\theta}$, i.e., for all $n$, there is at least one $K$ such that $s_{n} \in N_{K}$. The result is not required in what follows.)
 
\begin{lemma}\label{nkcomp}
$N_{K}$ is a sequence of complements.
\end{lemma}
\begin{proof}

 We prove that both $N_{K^{(-)}}$ and $N_{K^{(+)}}$ are complementary sequences. We begin with 
 $N_{K^{(-)}} = (s_{n} \mid \lfloor k\theta \rfloor \le s_{n} \le k \theta) $.
 Let $s^{\prime}_{n}$ represent the complement of $s_{n}$. That is, 
\begin{equation}\label{sumdef}
 s_{n}+s^{\prime}_{n} = \lfloor{k\theta}\rfloor + k\theta.
\end{equation}

Clearly if $s_{n} \in N_{K^{(-)}}$,
$
\lfloor k\theta \rfloor \le s^{\prime}_{n} \le k \theta.
$
Now we show that $s^{\prime}_{n} \in S_{\theta}$. Let $s_{n} = i+j\theta$, $s^{\prime}_{n} = i^{\prime}+j^{\prime}\theta$.
Then $\lfloor k\theta \rfloor \le i+j\theta \le k \theta$. This implies $\lfloor k\theta \rfloor = \lfloor i+j\theta \rfloor$, which gives 
 $i \le \lfloor k\theta \rfloor$ and $j \le k $.
Rewrite Equation (\ref{sumdef}) as 
\[
 i+i^{\prime} + (j+j^{\prime})\theta = \lfloor{k\theta}\rfloor + k\theta
\]
 
Because $\theta$ is irrational,
\[
\begin{array}{c}
0\le i+i^{\prime} = \lfloor{k\theta}\rfloor \\
0\le j+j^{\prime} = k.
\end{array}
\]
Then, $0\le i^{\prime} \le \lfloor{k\theta}\rfloor$ and $0\le j^{\prime} \le k$, which means $s^{\prime}_{n} = i^{\prime}+j^{\prime}\theta \in S_{\theta}$ and therefore $s^{\prime}_{n} \in N_{K^{(-)}}$ whenever $s_{n}\in N_{K^{(-)}}$. Then $N_{K^{(-)}}$ is a sequence of complements.

Next we prove the same for $N_{K^{(+)}}=(s_{n} \mid \lfloor k{(\theta}^{-1})\rfloor \theta \le s_{n} \le k )$.
Rewrite Equation (\ref{sumdef}) as

\begin{equation}\label{sumdef2}
 s_{n}+s^{\prime}_{n} = \lfloor k{(\theta}^{-1})\rfloor \theta + k.
\end{equation}
and so if $s_{n} \in N_{K^{(+)}}$, 
$
\lfloor k{(\theta}^{-1})\rfloor \theta \le s^{\prime}_{n} \le k .
$

If $s_{n} = i+j\theta \in N_{K^{(+)}}$ we have $i+j\theta <k$, and so $i \le k$. Dividing the inequality $\lfloor k{(\theta}^{-1})\rfloor \theta \le i+j\theta \le k$ through by $\theta$ we get $\lfloor \frac{k}{\theta}\rfloor \le \frac{i}{\theta} +j \le \frac{k}{\theta}$, and so $\lfloor \frac{k}{\theta}\rfloor =\lfloor \frac{i}{\theta} +j \rfloor $, or $j \le \lfloor \frac{k}{\theta}\rfloor$. Then
\[
\begin{array}{c}
0\le i+i^{\prime} = k \\
0\le j+j^{\prime} = \lfloor k{(\theta}^{-1})\rfloor.
\end{array}
\]

Therefore both $i^{\prime} \ge 0$ and $j^{\prime} \ge 0$, and so $s^{\prime}_{n} = i^{\prime}+j^{\prime}\theta \in S_{\theta}$ and $N_{K^{(+)}}$ is also a sequence of complements. Thus for all $K$, $N_{K}$ is a sequence of complements of $S_{\theta}$. Then, by Proposition \ref{palcomp}, $\Delta N_{K}$ is a palindrome. 
\end{proof}
We refer to $\Delta N_{K}$ as a {\it nuclear} palindrome. 

\subsection{Maximal sequences and palindromes}
In general, it is possible to expand a nuclear palindrome, adding elements symmetrically to the left and the right as we will demonstrate below. First we simplify our notation and let
\[
A =
\begin{cases}
 \min (N_{K^{(-)}}) \\
 \max (N_{K^{(+)}})
\end{cases}
B\theta =
\begin{cases}
 \max (N_{K^{(-)}}) \\
 \min (N_{K^{(+)}})
\end{cases}
\]

It follows that $A+B\theta = \min (N_{K}) + \max (N_{k})$ that is, complements in $N_{K}$ are defined with respect to $A+B\theta$. Furthermore, $A+B=K$. We can, in general, find longer complementary sequences of $A+B\theta$. Whenever $ i+j\theta \in N_{K}$, $0 \le i \le A$ and $0 \le j \le B$. However, these limits hold on $i$ and $j$ whenever $s_{n} < \max(N_{K+1})$. Then let 
\[
\begin{array}{ccl}
T_{K} &= &\left(s_{n} \mid \max(N_{K}) \le s_{n} < \max(N_{K+1})\right)\\
\vspace{1 pt}\\
T^{\prime}_{K} &=& \left(s_{n}\mid s^{\prime}_{n} \in T_{K} \right)
\end{array}
\]
where, again, complements are defined with respect to $A+B\theta$. Clearly, the sequences $(T_{K})$ and $(T^{\prime}_{K})$ have the same number of elements and $\min(T^{\prime}_{K}) + \max (T_{K}) = A+B\theta$.
Let $C_{K} = T^{\prime}_{K} \cup N_{k} \cup T_{K}$. We call $C_{K}$ a ``maximal'' sequence.

\begin{lemma}\label{ciscomp}
$C_{K}$ is a sequence of complements.
\end{lemma}

\begin{proof}
Like $N_{K}$, $C_{K}$ is a contiguous sequence of complements with respect to $A+B\theta$. If $s_{n} \in N_{K}$, then $s^{\prime}_{n} \in N_{K}$. If $s_{n} \in T_{K}$ (resp. $T^{\prime}_{K}$), then $s^{\prime}_{n} \in T^{\prime}_{K}$ (resp. $T_{K}$). 
\end{proof}
Consequently, by Proposition \ref{palcomp}, $\Delta C_{K}$ is a (maximal) palindrome.

\subsection{Medial subpalindromes}
For palindrome $p$, if $p = xp^{\prime} \widetilde{x}$ and $p^{\prime} \ne \epsilon$, then $p^{\prime}$ is a {\it medial subpalindrome} of $p$. (The term ``central factor'' is also used \cite[p. \ 2]{dro01}). It is easy to see that $p^{\prime}$ itself is a palindrome. Clearly, $\Delta (N_{K})$ is a medial subpalindrome of $\Delta (C_{K})$ because
\[\Delta (C_{K}) = \Delta T^{\prime}_{K} \cup \Delta (N_{K}) \cup \Delta T_{K} = \widetilde{\Delta T_{K}} \cup \Delta (N_{K}) \cup {\Delta T_{K}}.\]
 (The reversal of $\Delta T_{K}$ is  $\Delta T^{\prime}_{K}$.)

Let $\Delta M$ be a medial subpalindrome of $\Delta (C_{K})$ such that 
\begin{equation}\label{mlim}
 \Delta N_{K} \subseteq\Delta M \subseteq \Delta C_{K}.
\end{equation}

Let $S_{M} = \left(s_{m}\mid 0 \le s_{m} \le \max(M) \right )$. Then $\Delta S_{M}$ is a non-empty prefix of $\Delta S_{\theta}$ and $\max(M) > 0$.

\begin{lemma}\label{notc}
The palindrome $\Delta M$ is a ups of $\Delta S_{M}$.
\end{lemma}

\begin{proof}
Clearly, $\Delta N_{K}$ is a medial subpalindrome in $\Delta M$ and $\Delta M$ is a palindromic suffix of $\Delta S_{M}$. Therefore if $\Delta N_{K}$ is uni-occurrent in $\Delta S_{M}$ then $\Delta M$ is also. The maximal values for $i$ and $ j\theta$ in $S_{M}$ are $A$ and $B\theta$. Then $\Delta N_{K}$ is uni-occurrent in $\Delta S_{M}$ and $\Delta M$ is a ups in $\Delta S_{M}$.
\end{proof}

\begin{theorem}\label{lamisrich}
$\Lambda_{\theta}$ is a rich word.
\end{theorem}

\begin{proof}
From (\ref{mlim}) it follows that $\max(N_{K}) \le \max(M) \le \max(C_{K})$.
Then $\max (M)$ is any member of $T_{K}$. Furthermore, for any $s_{n} >0$, there is exactly one $T_{K}$ such that $s_{n} \in T_{K}$ because $\max(T_{K}) < \max(N_{K+1}) = \min(T_{K+1})$.
Therefore every (non-empty) prefix of $\Delta S_{M}$ of $\Delta S_{\theta}$ ends with ups $\Delta M$. The bijection $\lambda$ guarantees that $\Lambda_{\theta}$ is a rich word.
\end{proof}

\section{A projection of $\Lambda_{\theta}$ onto $\Gamma_{\theta}$, a word over a three-letter alphabet
}\label{3alphsec}

In Corollary \ref{3limit}, we found that any palindrome in $\Lambda_{\theta}$ is over a three-letter alphabet $\mathscr{X} = \{x,y,z\} \subset \mathscr{A}$. This suggests that we may map any Lambda word onto a word with only three symbols that preserves both palindromes and non-palindromes of the Lambda word. We construct a mapping $\gamma$ such that $\gamma(W) = \gamma(Z)$. We also require $\gamma(X) \ne \gamma(Y)$, where, as before $f(X)$ and $f(Y)$ form a member of the Hurwitz chain for $\theta$. Finally, both $\gamma(X)$ and $\gamma(Y)$ must be distinct from $\gamma(Z)$.

Let $A/B$ be a convergent of $\theta$ and let $a_{k}/b_{k}$ be the principal convergent that immediately precedes $A/B$. Then the pair $(A/B,a_{k}/b_{k})$ forms a member of the Hurwitz chain for $\theta$. We assign seed values to the first member of the chain $(1/1,1/0)$:
\[
\begin{array}{lllll}
\gamma |a_{-1}-b_{-1}\theta| &=& \gamma(1-0\theta) &=& 0\\
\gamma |a_{0}-b_{0}\theta| &=& \gamma(-1+1\theta) &=& 1
\end{array}
\]
and then define recursively:
\begin{equation}\label{3map}
\gamma|A-B\theta | = {3-(\gamma |a_{k}-b_{k}\theta |+\gamma| (A-a_{k}) - (B-b_{k})\theta |)}.
\end{equation}

Because $f(A/B) = f(a_{k}/b_{k}) \oplus f\left((A-a_{k})/(B-b_{k})\right)$, we can interpret (\ref{3map}) as $\gamma(Z) = {3 - (\gamma(Y) + \gamma(X)}$. 
We can rearrange (\ref{3map}) into
\begin{equation}\label{3map2}
\gamma|(A-a_{k}) - (B-b_{k})\theta | = {3-(\gamma|A-B\theta | +\gamma |a_{k}-b_{k}\theta |)}
\end{equation}
which may be construed as $\gamma(W) = {3-(\gamma(X) + \gamma(Y))}$.
The set $\{\gamma(X),\gamma(Y),\gamma(Z)\}$ maps bijectively onto $\{0,1,2\}$ and so does $\{\gamma(W),\gamma(X),\gamma(Y)\}$.
Therefore, 
\[
{\gamma(W) + \gamma(X) + \gamma(Y)} = {\gamma(X) + \gamma(Y) + \gamma(Z)} = 3.
\]
Moreover, for a fixed $(X,Y)$, $\gamma(W) = \gamma(Z)$ as required.
Letting $\vartheta = \log_{2} 3$, Figure \ref{into3} shows the projection of the infinite alphabet $\mathscr{A}$ of $\Lambda_{\vartheta}$ onto $\{0,1,2\}$.

\begin{figure}
\[
\begin{array}{|c|c|c|}
\hline
A/B & \lambda | A-B\vartheta | & \gamma|A-B\vartheta |\\
\hline
1/0 & 0 & 0\\
1/1 & 1 & 1\\
2/1 & 2 & 2\\
3/2 & 3 & 0\\
5/3 & 4 & 1\\
8/5 & 5 & 2\\
11/7 & 6 & 1\\
19/12 & 7 & 0\\
27/17 & 8 & 1\\
46/29 & 9 & 2\\
65/41 & 10 & 1\\
84/53 & 11 & 2\\
\hline
\end{array}
\]
\caption{The mappings $\lambda|A-B\vartheta|$ and $\gamma|A-B\vartheta|$.}
\label{into3}
\end{figure}

When $\theta = \phi$, the ``golden'' ratio, we can rewrite (\ref{3map}) and (\ref{3map2}) as
\[
\gamma |a_{k+1}-b_{k+1}\phi | = {3-(\gamma|a_{k}-b_{k}\phi | + \gamma |a_{k-1}-b_{k-1}\phi|)} = \gamma |a_{k-2}-b_{k-2}|\phi.
\]
Because $\gamma|a_{0}-b_{0}\theta| = 1$, it follows that $\gamma |a_{k}-b_{k}\phi| \equiv {k+1} {\pmod 3}$. That is, the successive values for $\gamma |A-B\phi|$ repeat the pattern (0, 1, 2).
With any other value of $\theta$, intermediate convergents may come between $a_{k}/b_{k}$ and $a_{k+1}/b_{k+1}$. If $\gamma |a_{k}-b_{k}\theta | = n$, then, under $\gamma$, the images of the intermediate convergents that follow alternate between ${n+1} {\pmod 3}$ and ${n-1} {\pmod 3}$. When $\theta$ is a quadratic irrational, the successive values for $\gamma |A-B\theta|$ are eventually cyclic, whereas other values of $\theta$ exhibit no cyclic patterns. This, of course, reflects the patterning of partial quotients in the continued fraction expansions of these values.

The mapping $\gamma$ partitions the set of convergents of $\theta$ into three classes, in contrast to the well-observed partitioning into two classes that depends on whether $A/B$ is greater than or less than $\theta$. The properties of this tripartite classification could warrant further investigation.

The mapping defined in (\ref{3map}) projects $\Lambda_{\theta}$ onto the {\it Gamma word generated by $\theta$}, $\Gamma_{\theta}$, a word defined over $\{0,1,2\}$. The table below illustrates that projection when  ${\theta} = \vartheta = \log_{2}3$:
\[
\begin{array}{| l | c c c c c c c c c c c c c c c c c |}
\hline
n & 0 & 1 & 2 & 3 & 4 & 5 & 6 & 7 & 8 & 9 & 10 & 11 & 12 & 13 & 14 & 15 & 16 \\
\hline
\Lambda_{\vartheta} & 0& 1& 2& 1& 2& 3& 2& 2& 3& 2& 3& 4& 3& 2& 3& 4& 3\\
\Gamma_{\vartheta} & 0& 1& 2& 1& 2& 0& 2& 2& 0& 2& 0& 1& 0& 2& 0& 1& 0\\
\hline
n & 17 & 18 &19 & 20 & 21 & 22 & 23 & 24 & 25 & 26 & 27 & 28 & 29 & 30 & 31 & 32 & 33 \\
\hline
\Lambda_{\vartheta} & 3& 4& 3& 4&   3& 3& 4& 3&   3& 5& 3& 3&   4& 3& 3& 5& 3\\
\Gamma_{\vartheta} & 0& 1& 0& 1&   0& 0& 1& 0&   0& 2& 0& 0&   1& 0& 0& 2& 0\\
\hline
\end{array}
\]

Although $\Gamma_{\theta}$ preserves palindromes and non-palindromes of $\Lambda_{\theta}$, $\Gamma_{\theta}$ is itself not rich: If any factor of a word is not rich, the word itself is not rich. The first six-letter factor of $\Gamma_{\vartheta}$, $012120$, has no uni-occurrent palindromic suffix and is therefore not rich, a property that $\Gamma_{\vartheta}$ inherits. Certainly there are recurrent factors in the Gamma word but it is not known if all factors of all Gamma words are recurrent.

\section{Acknowledgments}
The author wishes to thank the editor and anonymous readers for their
invaluable help. Thanks are also due to David Clampitt for many timely
suggestions and corrections and to the other members of Micrologus,
Emmanuel Amiot and Thomas Noll. Also, thanks to Sre\v{c}ko Brlek who
provided insight into full (or rich) words at a conference at IRCAM.

\bigskip
\hrule
\bigskip

\noindent 2013 {\it Mathematics Subject Classification} Primary 68R15; Secondary 06F05, 11J70, 11Y55.\\
\noindent{\it Keywords}: combinatorics on words, palindrome, rich word, continued fraction, Hurwitz chain. 

\bigskip
\hrule
\bigskip

\noindent(Concerned with sequences 
\seqnum{A022330},
\seqnum{A022331},
\seqnum{A167267},
\seqnum{A216448},
\seqnum{A216763}, and
\seqnum{A216764}.)

\bigskip
\hrule
\bigskip

\vspace*{+.1in}
\noindent
Received September 15 2012;
revised version received  January 18 2013; February 10 2013.
Published in {\it Journal of Integer Sequences}, March 2 2013.

\bigskip
\hrule
\bigskip

\noindent
Return to
\htmladdnormallink{Journal of Integer Sequences home page}{http://www.cs.uwaterloo.ca/journals/JIS/}.
\vskip .1in


\begin{thebibliography}{99}
\bibitem{all03}{J.-P. Allouche, M. Baake, J. Cassaigne, and D. Damanik,
Palindrome complexity. {\it Theoret. Comput. Sci.} {\bf{292}} (2003),
9--31.}

\bibitem{bea27}{S. Beatty, A. Ostrowski, J. Hyslop, and A. C. Aitken,
Solutions to Problem 3173. {\it Amer. Math. Monthly} {\bf{34}} (1927),
159--160.}

\bibitem{ber02}{J. Berstel and P. S\'{e}\'{e}bold, Sturmian words, in
{\it Algebraic Combinatorics on Words}, Encyc. Math. Appl., Vol.\ 90,
Cambridge Univ. Press, 2002, pp. 45--110.}

\bibitem{ber08}{J. Berstel, A. Lauve, C. Reutenauer, and F. V. Saliola,
{\it Combinatorics on Words: Christoffel Words and Repetitions in
Words}, CRM Monograph Series, Vol.\ 27, Amer. Math. Soc., 2008.}

\bibitem{blo11}{A. Blondin Mass\'{e}, S. Brlek, S. Labb\'{e}, and L.
Vuillon, Palindromic complexity of codings of rotations, {\it Theoret.
Comput. Sci.} {\bf{412}} (2011), 6455--6463.}

\bibitem{brl04}{S. Brlek, S. Hamel, M. Nivat, and C. Reutenauer: On the
palindromic complexity of infinite words, {\it Internat. J. Found.
Comput. Sci.} {\bf{15}} (2004), 293--306.}

\bibitem{buc09a}{M. Bucci, A. de Luca, and A. De Luca, Rich and
periodic-like words, in {\it DLT 2009, Proceedings of the 13th
International Conference on Developments in Language Theory}, Lect.
Notes in Comput. Sci., Vol.\ 5583, Springer, 2009, pp.\ 145--155.}

\bibitem{buc09b}{M. Bucci, A. De Luca, A. Glen, and L. Q. Zamboni, A
new characteristic property of rich words, {\it Theoret. Comput. Sci.}
{\bf{410}} (2009), 2860--2863.}

\bibitem{car11}{N. Carey, On a class of locally symmetric sequences:
The right infinite word $\Lambda_{\theta}$, in {\it Mathematics and
Computation in Music}, Lect. Notes in Comp. Sci., Vol.\ 6726, Springer,
2011, pp.\ 42--55.}

\bibitem{car96}{N. Carey and D. Clampitt, Regions: A theory of tonal
spaces in early medieval treatises, {\it J. Music Th.} {\bf{40}}
(1996), 113--147.}

\bibitem{car12}{N. Carey and D. Clampitt, Two theorems concerning
rational approximations, {\it J. Math. Mus.} {\bf{6}} (2012), 61--66.}

\bibitem{carp05}{A. Carpi and A. de Luca, Central sturmian words:
Recent developments, in {\it Developments in Language Theory}, Lect.
Notes in Comp. Sci., Vol. 3572, Springer, 2005, pp.\ 36--56. }

\bibitem{dro01}{X. Droubay, J. Justin, and G. Pirillo, Episturmian
words and some constructions of de Luca and Rauzy, {\it Theoret.
Comput. Sci.} {\bf{255}} (2001), 539--553.}

\bibitem{dro95}{X. Droubay, Palindromes in the Fibonacci word, {\it
Inform. Process. Lett.} {\bf{55}} (1995), 217--221.}

\bibitem{dro99}{X. Droubay, and G. Pirillo, Palindromes and Sturmian
words, {\it Theoret. Comput. Sci.} {\bf{223}} (1999), 73--85.}

\bibitem{gar89}{M. Gardner, {\it Penrose Tiles and Trapdoor Ciphers
\ldots and the Return of Dr. Matrix}, W. H. Freeman, 1989.}

\bibitem{gol88}{J. Goldman, Hurwitz sequences, the Farey process, and
general continued fractions, {\it Adv. Math.} {\bf{72}} (1988),
239--260.}

\bibitem{gle09a}{A. Glen, and J. Justin, Episturmian words: a survey,
{\it RAIRO Theor. Inform. Appl.} {\bf{43}} (2009), 402--433.}

\bibitem{gle09b}{A. Glen, J. Justin, S. Widmer, and L. Q. Zamboni,
Palindromic richness, {\it European J. Combin.} {\bf{30}} (2009),
510--531.}

\bibitem{gra94}{R. L. Graham, D. E. Knuth, and O. Patashnik, {\it
Concrete Mathematics: A Foundation for Computer Science}, 2nd ed.,
Addison-Wesley, 1994.}

\bibitem{hur94}{A. Hurwitz, \"Uber die angen\"aherte Darstellung der
Zahlen durch rationale Br\"uche, {\it Math. Ann.} {\bf{44}} (1894),
417--436.}

\bibitem{jon55}{B. Jones, {\it The Theory of Numbers}, Holt, Rinehart,
and Winston, 1955.}

\bibitem{kim93}{C. Kimberling, Interspersions and dispersions, {\it
Proc. Amer. Math. Soc.} {\bf{117}} (1993), 313--321.}

\bibitem{kim97}{C. Kimberling, Fractal sequences and interspersions,
{\it Ars Combin.} {\bf{45}} (1997), 157--168.}

\bibitem{kim04}{C. Kimberling and J. Brown, Partial complements and
transposable dispersions, {\it J. Integer. Seq.} {\bf{7}} (2004),
\href{https://cs.uwaterloo.ca/journals/JIS/VOL7/Kimberling/kimber67.html}{Article
04.1.6.}}

\bibitem{khi35}{A. Khinchin, {\it Continued Fractions}, trans.\ ed.\ by
H. Eagle. Original Russian edition, 1935. Univ.\ Chicago Press, 1964. }

\bibitem{lot02} M. Lothaire, {\it Algebraic Combinatorics on Words},
Encyc. Math. Appl., Vol.\ 90, Cambridge Univ. Press, 2002.

\bibitem{luc08}{A. Luca, A. Glen, and L. Q. Zamboni, Rich, Sturmian,
and trapezoidal words, {\it Theoret. Comput. Sci.} {\bf{407}} (2008),
569--573.}

\bibitem{old63}{ C. D. Olds, {\it Continued Fractions}, Random House,
1963.}

\bibitem{ric81}{I. Richards, Continued fractions without tears, {\it
Math. Mag.} {\bf{54}} (1981) 163--171. }

\bibitem{} {N. J. A. Sloane, The On-Line Encyclopedia of Integer
Sequences, \url{http://oeis.org}.}

\bibitem{ste58}{M. A. Stern, \"{U}ber eine zahlentheoretische Funktion.
{\it J. reine angew. Math.} \bf{55}} (1858) 193--220.
\end{thebibliography}
\end{document}